\documentclass[10pt]{amsart}
\usepackage[margin=4cm]{geometry}

\usepackage{color} 

\usepackage{amsmath,amssymb}
 \usepackage{appendix}

\makeatletter
\let\@wraptoccontribs\wraptoccontribs
\makeatother

\newtheorem{theorem}{Theorem}[section]
\newtheorem{prop}[theorem]{Proposition}
\newtheorem{proposition}[theorem]{Proposition}

\newtheorem{corollary}[theorem]{Corollary}
\newtheorem{example}[theorem]{Example}

\theoremstyle{definition}

\theoremstyle{remark}
\newtheorem{remark}[theorem]{Remark}

\newtheorem*{claim*}{Claim}

\newcommand{\bd}{{\partial}}

\newcommand{\reals}{{\mathbb R}}

\newcommand{\proj}{{\mathbb P}}

\newcommand{\mC}{{\mathcal{C}}}
\newcommand{\mR}{{\mathcal{R}}}

\newcommand{\vol}{{\rm vol}}

\newcommand{\cN}{{\mathcal{N}}}

\newcommand{\cD}{{\mathcal{D}}}
\newcommand{\cH}{{\mathcal{H}}}

\renewcommand{\part}[2]{\frac{\partial #2}{\partial x_{#1}}}

\numberwithin{equation}{section}

\begin{document}

\title[Conformal invariants II. Manifolds with boundary]{Conformal invariants from 
nodal sets II. Manifolds with boundary}
 
\author[G. Cox]{Graham Cox}
\address{Department of Mathematics and Statistics, Memorial University of Newfoundland, St. John's, NL A1C 5S7, Canada}
\email{gcox@mun.ca}

\author[D. Jakobson]{Dmitry Jakobson}
\address{Department of Mathematics and
Statistics, McGill University, Montr\'eal, Ca\-na\-da.}
\email{jakobson@math.mcgill.ca}

\author[M. Karpukhin]{Mikhail Karpukhin}
\address{Department of Mathematics,
University of California, Irvine, 340 Rowland Hall, Irvine, CA 92697-3875
}
\email{mkarpukh@uci.edu}

\author[Y. Sire]{Yannick Sire} 
\address{Department of Mathematics, Johns Hopkins University, 404 Krieger Hall
3400 N. Charles Street, Baltimore, MD 21218}
\email{sire@math.jhu.edu}



\thanks{G.C. acknowledges the support of NSERC grant RGPIN-2017-04259. 
D.J.\ was supported by NSERC and FQRNT grants and Peter Redpath Fellowship of McGill 
University (Canada).  
M.K.\ was supported by Schulich Fellowship of McGill University (Canada) at initial stages of this project.}

\begin{abstract}
In this paper, we study conformal invariants that arise from nodal sets and negative 
eigenvalues of conformally covariant operators on manifolds with boundary.  We also 
consider applications 
to curvature prescription problems on manifolds with boundary.  We relate Dirichlet and 
Neumann eigenvalues and put the results developed here for the Escobar problem into 
the more general framework of boundary operators of arbitrary order. 
\end{abstract}

\maketitle
\tableofcontents

\section{Introduction} 

The Yamabe problem on a closed Riemannian manifold is to find a metric of constant scalar curvature 
in a given conformal class; it generalizes the uniformization theorem for Riemann surfaces.  
It was solved by Yamabe, Trudinger, Aubin and Schoen.  
A related problem is to {\em prescribe} scalar curvature in a given conformal class; the two-dimensional 
version (prescribing Gauss curvature on a surface) is the {\em Nirenberg problem.}  The 
problem of prescribing scalar curvature {\em without} fixing a conformal class was solved by Kazdan 
and Warner \cite{KW75} (the solution is known as the {\em trichotomy theorem}).  Fixing a conformal class makes the 
problem more difficult.  

The conformal Laplacian appears naturally in the study of the Yamabe and curvature prescription 
problems.  Nonzero eigenfunctions in its kernel provide obstructions to prescribing curvature; 
see \cite{CGJP1,CGJP2} for precise statements.  In addition to the conformal Laplacian, 
in the papers \cite{CGJP1,CGJP2,GHJL} the authors considered general conformally covariant operators 
(including the Paneitz operator), and studied negative and zero eigenvalues 
of  those operators on manifolds without boundary.  They also considered 
applications to the associated curvature prescription problems (including the $Q$-curvature), and
 constructed conformal invariants determined by eigenfunctions in the 
kernel of the corresponding operator.  In the current paper, we aim to generalize 
those results to manifolds 
with boundary.

The Yamabe problem on manifolds with boundary was solved in a series of works by  
Escobar \cite{Escobar:1992:JDG,Escobar:1992:Annals}, Marques 
\cite{Marques}, Chen \cite{Chen} and Mayer-Ndiaye \cite{MN}.  
They proved  
that on any compact manifold with boundary of dimension 
$n\geq 3$, there exists in every conformal class a scalar flat metric in the interior with constant 
mean curvature of the boundary.  

The problem of {\em prescribing} scalar curvature $R_g$ on the manifold $M$ 
and mean curvature $h_g$ on the boundary $\bd M$ was considered 
in \cite{Escobar:2003:JFA} and other papers.  It was shown in \cite[Cor. 2]{Escobar:2003:JFA} 
that the solution attaining $(R_g,h_g)$ is unique in the conformal class of $[g]$ if $R_g\leq 0$ and $h_g\leq 0$.

In the current paper we consider general conformally covariant operators on manifolds with 
boundary, and generalize the conformal invariants constructed in \cite{CGJP1,CGJP2} to 
this setting.  We extend  the results of \cite{GHJL}, showing that $0$ is generically 
{\em not} an eigenvalue in certain natural boundary value problems.  We also answer 
some natural questions about the number of negative eigenvalues for conformally covariant 
boundary value problems.

\subsection*{Main results}
We summarize below several important results of the present paper. 
The first theorem extends the main result of \cite{GHJL} to manifolds with boundary: 

\begin{theorem}\label{not0}
Generically, $0$ is not an eigenvalue of the conformal Laplacian on a manifold with boundary.
\end{theorem}
The proof follows easily from Proposition \ref{generic:thm} of Section \ref{sec:no:zero} 
by methods of \cite{GHJL}.  

The next result shows that the number of negative eigenvalues of the 
conformal Laplacian on manifolds with boundary can be made arbitrarily large: 
\begin{theorem}\label{thm:many:negative}
For any compact manifold with boundary and any natural number $m$, there exists a 
conformal class $\mC$ such that for any metric $g\in \mC$ one has 
$\lambda_m(P_{1,g}^R)<0$, where the operator $P_{1,g}^R$ is defined below in \eqref{ConfN:def}. 
\end{theorem}
The proof in section \ref{sec:many:negative} uses the ideas of El Sayed~\cite{Elsayed}.

The next result provides an obstruction to prescribing curvature in a conformal class 
$[g]$ with non-trivial $\ker(P_{1,g}^R)$, formulated as an inequality for a {\em single 
nodal domain} of a function $u\in\ker(P_{1,g}^R)$.  
\begin{theorem}\label{obstructionNodal}
Let $(Q,f)\in \mR([g])$ be the scalar curvature and the mean curvature of some metric in 
the conformal class $[g]$. Then, there is a pair of positive functions $\omega_i\in 
C^\infty(M)$, $\omega_b\in C^\infty(\bd M)$ such that for any nonzero $u\in\ker(P_{1,g}^R)$ 
and any nodal domain $\Omega$ of $u$,
$$
\int\limits_\Omega Q|u|\omega_i \,dv_g + \int\limits_{\bd\Omega\cap\bd M}f|u|\omega_b\,d\sigma_g<0.
$$
\end{theorem}

Finally, Theorem \ref{prop:negative:eig:count} (the main result in Section \ref{sec:DNinequalities}) 
provides a generalization of a result in 
\cite{Friedlander} to conformally covariant operators, relating the number of negative 
eigenvalues for the Dirichlet and Robin problems to the number of negative eigenvalues for a 
conformally covariant version of the Dirichlet-to-Neumann map.

\subsection*{Outline of the paper.}

In section \ref{sec:BVP} we give examples of conformally covariant boundary value 
problems and describe natural extensions of the results in 
\cite{CGJP1,CGJP2} to manifolds with boundary.  
In section \ref{sec:no:zero}, we show that generically the conformal Laplacian on manifolds 
with boundary does not have the eigenvalue $0$.  In section \ref{sec:many:negative}, we show 
that the number of negative eigenvalues of the conformal Laplacian on manifolds with 
boundary can be arbitrarily large.  In section \ref{sec:curvature:prescription} 
we explore some natural curvature prescription problems (in a given conformal class) on manifolds 
with boundary;  we also study the relationship between nodal sets of eigenfunctions in the kernel of the 
conformal Laplacian in the interior of the manifold, 
and the corresponding nodal sets on the boundary.  
In section \ref{sec:DNinequalities}, we study inequalities between Dirichlet and 
Robin eigenvalues, and establish Friedlander-type results in the conformal setting.  
In section \ref{sec:Poincare:Einstein} we study a continuous family of GJMS operators 
and their conformal extensions to the boundary of Poincare-Einstein manifolds.  
In section \ref{sec:further:questions} we outline some natural questions that we have not considered 
in this paper.


\section{Conformally covariant boundary value problems}\label{sec:BVP}

A key observation used in \cite{CGJP1,CGJP2} (and many other papers) is that 
eigenfunctions that belong to the kernel of a conformally covariant operator 
(on a manifold without boundary) are multiplied by a positive function under 
a conformal change of the metric.  As a consequence, their nodal sets and 
nodal domains are preserved; other consequences of this fact were explored 
in \cite{CGJP1,CGJP2}.  

Here we consider a manifold $M$ (of dimension $n\geq 3$) 
with smooth boundary $\bd M$.
Given a conformally covariant 
operator $P$ on $M$, we consider eigenfunctions $P\phi = \lambda\phi$ 
on $M$.  
We are interested in boundary value problems for $P$ where eigenfunctions 
corresponding to $\lambda=0$ transform 
in the same way as when $\bd M=\emptyset$.

We start with a discussion of the conformal Laplacian,
\begin{equation}\label{def:conflap}
P_{1,g}=\Delta_g+\frac{n-2}{4(n-1)}R_g, 
\end{equation}
where $\Delta_g$ is the positive-definite Laplacian for $g$, and $R_g$ is the scalar 
curvature.  If $\hat g=e^{2\omega}g$ is a metric in the conformal class 
$[g]$, then $P_{1,g}$ transforms according to the law 
$$
P_{1,\hat g}=e^{-(n/2+1)\omega}P_{1,g} e^{(n/2-1)\omega}.
$$
Accordingly, 
\begin{equation}\label{kernel:change}
\ker P_{1,\hat g}=e^{-(n/2-1)\omega}\ker P_{1,g}.
\end{equation} 

Consider now a boundary value problem for $P_{1,g}$ on $(M,\bd M)$.  In the interior 
we have $P_{1,g}\phi = \lambda\phi$.  Next, we need to specify 
boundary conditions on $\bd M$.  We would like the transformation formula 
\eqref{kernel:change} to hold for eigenfunctions with $\lambda=0$.  

Standard boundary conditions include {\em Dirichlet BC} 
\begin{equation}\label{Dirichlet:def}
\left\{
\begin{aligned}
P_{1,g} \phi(x) &=\lambda\phi(x), \quad x\in M\\
\phi(x) &= 0, \quad x\in \bd M
\end{aligned}
\right.
\end{equation}
and {\em Neumann BC} 
\begin{equation}\label{Neumann:def}
\left\{
\begin{aligned}
P_{1,g} \phi(x) &=\lambda\phi(x), \quad x\in M\\
\bd_\nu\phi(x) &= 0, \quad x\in \bd M.  
\end{aligned}
\right.
\end{equation}
Here $\nu = \nu_g$ is the unit (with respect to the metric $g$) outer normal vector field on $\bd M$. We remark that the direction of $\nu$ is preserved under conformal transformations.

\begin{example}
If $0$ is an eigenvalue of $P_{1,g}$ on a closed manifold, then it is an eigenvalue of the Dirichlet problem on any nodal domain of the corresponding eigenfunction.
\end{example}

We next state the first result about boundary value problems  
\begin{prop}\label{conformal:BC}
Let $\hat g=e^{2\omega}g$ be two metrics on $M$ lying in the 
same conformal class.
\begin{itemize}
\item[(a)] Solutions of the Dirichlet BVP \eqref{Dirichlet:def} with $\lambda=0$ 
transform under the rule \eqref{kernel:change}.
\item[(b)] Solutions of the Neumann BVP \eqref{Neumann:def} with $\lambda=0$ 
transform under the rule \eqref{kernel:change},
provided $\omega$ satisfies $\bd_\nu\omega(x)=0$ for all $x\in\bd M$.    
\end{itemize}
\end{prop}

\begin{proof}  The condition \eqref{Dirichlet:def} is clearly invariant 
under the transformation \eqref{kernel:change}.  For the condition 
\eqref{Neumann:def}, we remark that 
$$
\bd_\nu\left(e^{-(n/2-1)\omega}\phi\right)=
e^{-(n/2-1)\omega}\left((1-n/2)\bd_\nu\omega+\bd_\nu\phi\right)
$$
and both expressions in the bracket vanish if $\bd_\nu\omega=0$.  
\end{proof}

It follows that choosing boundary conditions as in Proposition 
\ref{conformal:BC} allows one to develop the theory analogous to 
that in \cite{CGJP1,CGJP2}. In particular, the following Proposition can be proved easily by the methods of 
\cite{CGJP1}.  
\begin{proposition}\label{prop:nodal-inv-bdry}
Let $P$ be a conformally covariant operator on $(M,\bd M)$ whose kernel transforms 
according to \eqref{kernel:change}.  Let $g$ be a Riemannian metric such that 
$\dim \ker P_g\geq 1$.  Then the following quantities give rise 
to invariants of the conformal class $[g]$.    
\begin{enumerate}
\item The dimension $k$ of $\ker P_g$  
\item The number of negative eigenvalues of $P_g$  
\item Nodal sets $\cN(u)$ and nodal domains of any nonzero eigenfunction $u\in\ker P_{g}$ 
\item Non-empty intersections of nodal sets of eigenfunctions in $\ker P_g$ and their 
complements, provided $k\geq 2$
\item Let $k\geq 2$, and let $(u_1,\ldots,u_k)$ be a basis of $\ker P_g$.  
Let $\widetilde{M}=M\setminus\left(\cap_{i=1}^k  \cN(u_i)\right)$.  Define  
$\Phi_g:\widetilde{M}\to \reals\proj^{k-1}$  by 
$\Phi_g(x)=(u_1(x):\ldots :u_k(x))$.  Then the orbit of $\Phi_g(\widetilde{M})$ in 
$\reals\proj^{k-1}$ under the action of $GL_k(\reals)$ is conformally invariant.  
\item The same results hold for $\bd M$ and the restrictions of eigenfunctions to $\bd M$.  
\item Let $p=2n/(n-2)$, and let $u\in\ker P_g$.  Then the density $|u(x)|^pd\vol(x)$ is conformally invariant.  
In particular, the following quantities are conformally invariant: $||u||_p$, as well as 
$\int_\Omega |u(x)|^pd\vol(x)$, where $\Omega$ is any domain in $M$.  Similar results hold 
for restrictions of $u$ to submanifolds $N$ of $M$, with suitably adjusted exponents $p=p(N)$.  
\end{enumerate}
\end{proposition}

Next, consider the boundary operator $B_g$ given by 
$$B_gu(x) = \bd_{\nu_g} u (x) + 
\frac{n-2}{2(n-1)}h_g u(x),$$
 where $h_g = tr_g II$ is the mean curvature of the boundary. We have

\begin{prop}
\label{confBVP}
The operator $B_g$ is conformally covariant, i.e. 
$$
B_{e^{2\omega}g} \left(e^{(-\frac{n}{2}+1)\omega}u\right) = e^{-\frac{n}{2}\omega}B_g u.
$$
\end{prop}
\begin{proof}
By definition, the second fundamental form is given by $II_{ij} = \frac{1}{2}\bd_\nu g_{ij}$. 
Thus, one has the following formula for the mean curvature
\begin{equation}
\label{mean}
h_g =  \frac{1}{2}(\bd_{\nu_g} g_{ij})g^{ij}.
\end{equation}
Note that $\nu_{e^{2\omega}g} = e^{-\omega}\nu_g$. Therefore, formula~(\ref{mean}) implies
\begin{equation*}
\begin{split}
h_{e^{2\omega}g} &= \frac{1}{2}e^{-3\omega}(\bd_{\nu_g}(e^{2\omega}g_{ij}))g^{ij} = 
e^{-\omega}(\bd_{\nu_g}\omega) g^{ij}g_{ij} + \frac{1}{2}e^{-\omega}(\bd_{\nu_g}g_{ij})g^{ij}\\ &=
 e^{-\omega}((n-1)\bd_{\nu_g}\omega + h_g).
\end{split}
\end{equation*}
Applying this transformation law to $B_g$, we obtain
\begin{equation*}
\begin{aligned}
B_{e^{2\omega}g}(e^{(-\frac{n}{2}+1)\omega}u) &= 
e^{-\omega}\bd_{\nu_g}(e^{(-\frac{n}{2}+1)\omega}u) +\\ 
\; &\; e^{-\omega}\frac{n-2}{2(n-1)}((n-1)\bd_{\nu_g}\omega + h_g)e^{(-\frac{n}{2}+1)\omega}u = 
e^{-\frac{n}{2}\omega} B_{g}u.
\end{aligned}
\end{equation*}
\end{proof}

Proposition~\ref{confBVP} allows us to define two types of conformally covariant eigenvalue 
problems on $M$. 

 \vspace{0.5cm}
 
{\em Conformal Robin BVP}:
\begin{equation}\label{ConfN:def}
\left\{
\begin{aligned}
P_{1,g}\phi(x) &=\lambda \phi(x), \quad x\in M\\
B_g\phi(x) &= 0, \quad x\in \bd M.  
\end{aligned}
\right.
\end{equation}
Proposition~\ref{confBVP} implies that solutions of the Conformal Robin BVP~(\ref{ConfN:def}) 
with $\lambda = 0$ transform under the rule \eqref{kernel:change}. In the following we 
denote the operator $P_{1,g}$ with conformal Robin BC as $P_{1,g}^R$.

\vspace{0.5cm}

{\em Conformal Steklov BVP (Conformal Dirichlet-to-Robin map)}
Suppose that $\lambda=0$ is {\em not} an eigenvalue of the Dirichlet BVP \eqref{Dirichlet:def}. Then for any $u\in C^\infty(\bd M)$ there exists a unique solution $\cH_g(u)$ (see e.g. \cite[\S 1.3]{LLS}) 
to 
\begin{equation}\label{ConfDtN:def}
\left\{
\begin{aligned}
P_{1,g}(\cH_g(u))(x) &=0, \quad x\in M\\
\cH_g(u) &= u, \quad x\in \bd M.  
\end{aligned}
\right.
\end{equation}
We define the {\em Conformal Dirichlet-to-Robin map} $\cD_g\colon C^\infty(\bd M)\to 
C^\infty(\bd M)$ as 
\begin{equation}\label{Dirichlet:to:Robin:def}
\cD_g(u) = B_g(\cH_g(u)).
\end{equation}
We have then
\begin{prop}
The operator $\cD_g$ is conformally covariant, i.e.
$$
\cD_{e^{2\omega}g}\left(e^{(-\frac{n}{2}+1)\omega}u\right) = e^{-\frac{n}{2}\omega}\cD_g(u). 
$$
\end{prop}
\begin{proof}
Conformal covariance of $P_{1,g}$ implies that 
$$
\cH_{e^{2\omega}g}(e^{(-\frac{n}{2}+1)\omega}u) = e^{(-\frac{n}{2}+1)\omega}\cH_g(u).
$$
Proposition~\ref{confBVP} and the definition of $\cD_g$ then imply
$$
\cD_{e^{2\omega}}(e^{(-\frac{n}{2}+1)\omega}u) = 
B_{e^{2\omega}g}\left(e^{(-\frac{n}{2}+1)\omega}\cH_g(u)\right) = 
e^{-\frac{n}{2}\omega}B_g(\cH_g(u)) = e^{-\frac{n}{2}\omega}\cD_g(u).
$$
\end{proof}

\begin{remark}
It is easy to see that $(\ker P_{1,g}^R)|_{\bd M} = \ker\cD_g$. Also, 
the number of negative eigenvalues of $\cD_g$ is equal to the number of negative eigenvalues 
of the Robin boundary value problem \eqref{ConfN:def} minus the number of negative eigenvalues for 
the Dirichlet problem \eqref{Dirichlet:def}. This is proved in 
Theorem \ref{prop:negative:eig:count}. Both of these claims remain true when 0 is a 
Dirichlet eigenvalue, provided $\cD_g$ is replaced by the operator $\widehat\cD_g$ defined below.
\end{remark}

If 0 is an eigenvalue of the Dirichlet problem \eqref{Dirichlet:def}, then
$\ker P_{1,g}^D$ is nontrivial. Define
\[
	S = \left\{ u \in C^\infty(\bd M) : \int_{\bd M} u \bd_\nu w \,d\sigma_g = 0 
	\text{ for all } w \in \ker P_{1,g}^D \right\}.
\]
The boundary value problem \eqref{ConfDtN:def} has a solution 
for any $u \in S$. Moreover, the solution is unique if we require it to be orthogonal to 
$\ker P_{1,g}^D$ in $L^2(M)$. Denoting this unique solution by $\widehat\cH_g(u)$, we 
define the operator $\widehat \cD_g \colon S \to S$ by
\begin{equation}\label{conformal:DtoN}
	\widehat \cD_g u = \Pi_S B_g(\widehat\cH_g(u)),
\end{equation}
where $\Pi_S$ is the $L^2(\bd M)$-orthogonal projection onto $\bar S$. If 
$w \in \ker P_{1,g}^D$, then $B_g w = \bd_\nu w \in S^\perp$, so $\Pi_S(B_g w) = 0$ and we in fact have
\[
	\widehat \cD_g u = \Pi_S B_g (\psi)
\]
where $\psi$ is \emph{any} solution to the boundary value problem \eqref{ConfDtN:def}. 
Alternatively, one can easily show that there exists a unique solution $\widetilde\cH_g(u)$ 
to \eqref{ConfDtN:def} with the property that $B_g(\widetilde\cH_g(u)) \in S$, and so
\[
	\widehat \cD_g u = B_g(\widetilde\cH_g(u)).
\]

\section{Generically there is no zero eigenvalue: proof of Theorem \ref{not0}}\label{sec:no:zero}

Let $g$ be a metric such that $\dim\ker P_{1,g}^R = k>0$. For simplicity we assume $R_g\equiv -1$ and $h_g\equiv 0$.
The goal of this section is the following proposition. 
\begin{proposition}
\label{generic:thm}
There exists $g_t = g + t\tilde g$ such that for all small enough $t\ll 1$ one has $\dim\ker P_{1,g_t}^R <k$. 
\end{proposition}
Repeating the deformation inductively one gets that for a generic metric, $0$ is not an eigenvalue of $P_{1,g}^R$. Our proof closely follows~\cite{GHJL}, where the same statement is proved for closed manifolds.

Our main observation is that if $\tilde g\equiv 0$ in a neighbourhood of $\partial M$, then the 
computations of~\cite[\S 4]{GHJL} follow through verbatim. Let $E_0=\ker P_{1,g}^R$ be the 
$0$-eigenspace and let $c_n = \frac{n-2}{4(n-1)}$. Consider the operator $Q_{g,\tilde g} = \Pi_{E_0}\frac{d}{dt}|_{t=0}(\Delta_{g_t} + c_n R_{g_t})\big|_{E_0}$ on $E_0$. Then Proposition~\ref{generic:thm} is equivalent to the statement that 
$Q_{g,\tilde g}\not\equiv 0$.

In~\cite{GHJL} one obtains that if $\mathrm{tr}_g \tilde g=0$, then
$$
(Q_{g,\tilde g}\psi,\psi) = (\tilde g,c_n\psi(2\mathrm{Hess}^{\circ}(\psi) - \psi\mathrm{Ric}^{\circ}) 
+(2c_n-1)(d\psi\otimes d\psi)^\circ),
$$
where $\circ$ denotes the traceless part of a bilinear form.

Assume that $q=c_n\psi(2\mathrm{Hess}^{\circ}(\psi) - \psi\mathrm{Ric}^{\circ}) +(2c_n-1)(d\psi\otimes d\psi)^\circ\not\equiv 0$. Then there exists $p\in\mathrm{int}M$ such that $q(p)\ne 0$. Let 
$p\in U\subset \mathrm{int}M$ be a neighbourhood of $p$ and let $\rho_\varepsilon\in C^\infty_0(U)$ be 
a cut-off function such that $\rho_\varepsilon(p)= 1$. Set $\tilde g = \rho_\varepsilon q$, then  
$(Q_{g,\tilde g}\psi,\psi) > 0$ for $\varepsilon\ll1$, i.e. $Q_{g,\tilde g}\not \equiv 0$.

Assume now that $q\equiv 0$. Then, since $R_g\equiv -1$ and $h_g\equiv 0$, $\psi$ is an eigenfunction 
of $\Delta_g$ with eigenvalue $c_n$, i.e it is orthogonal to constant function. As a result, the nodal set 
$N(\psi)$ intersects $\operatorname{int}M$ nontrivially. Let us restrict $q=0$ to $N(\psi)\cap\operatorname{int}M$. Then, by 
definition of $q$ one obtains $(d\psi\otimes d\psi)^\circ = 0$ there. The latter is equivalent to 
$$
d\psi\otimes d\psi = \frac{1}{n}|\psi|_g^2g,
$$
which implies $d\psi\equiv 0$. This is impossible by, e.g. 
\cite{Ha,HHL,HHHN}, as in \cite{GHJL}.


\section{Number of negative eigenvalues: proof of Theorem \ref{thm:many:negative}}\label{sec:many:negative}

The proof is very similar to the proof of Proposition 3.2 in the paper of 
El Sayed~\cite{Elsayed}. Let $g_0$ be some metric on $M$. The Rayleigh quotient 
of $P_{1,g}^R$ is given by
$$
R_g[u] = \frac{\int_M |\nabla u|^2\,dv_g + \frac{n-2}{4(n-1)}\int_M R_gu^2\,dV_g + 
\frac{n-2}{2(n-1)}\int_{\partial M}h_gu^2\,d\sigma_g}{\int_M u^2\, dV_g}.
$$

The plan is to locally change the metric $g_0$ to a new metric $g_p$ around a point $p$ to 
create a function $u_p$ supported in the neighbourhood of $p$ such that $R_{g_p}[u_p]<0$. 
Doing so at $m$ different points creates an $m$-dimensional negative subspace for the Rayleigh quotient.

Let $p_i$, $i=1,\ldots, m$ be a collection of points in the interior of $M$ and let $p_i\in U_i$ 
be sufficiently small pairwise disjoint geodesic balls lying in the interior of $M$. We identify 
each $U_i$ with a copy of $\mathbb{S}^n\backslash\{p\}$ for some fixed point $p\in\mathbb{S}^n$. 
By a result of Aubin~\cite{Au} there exists metric $g$ on $\mathbb{S}^n$ such that the 
Yamabe invariant $Y_g(\mathbb{S}^n)<0$. Therefore there exists a function 
$u\in C^\infty(\mathbb{S}^n)$ such that $R_g[u]<-2\varepsilon$ for some $\varepsilon>0$. 
Let $V_\delta\subset V_{2\delta}\subset \mathbb{S}^n$ be two small balls around $p$ small 
enough so that there exists a cut-off function $\varphi_\delta$ with $\varphi_\delta\equiv 1$ on 
$\mathbb{S}^n\backslash V_{2\delta}$, $\varphi_\delta = 0$ on $V_\delta$, 
$0\leqslant\varphi_\delta\leqslant 1$ and $||\nabla\varphi_\delta||^2_{L_g^2(\mathbb{S}^n)}\leqslant \delta$. 
As $\dim M\geqslant 3$, such balls exist for all $\delta>0$. Then one has 
$R_g[u\varphi_\delta]\to R_g[u]$ as $\delta\to 0$.

Thus, for small enough $\delta>0$ one has $R_g[u\varphi_\delta]<-\varepsilon$. We let 
$u_i$ be the functions defined as $u\varphi_\delta$ in the neighbourhood $U_i$ of $p_i$. 
We define the metric $g_k$ on $M$ to be any metric equal to $g$ on $U_i\backslash V_{\delta}$, 
where $V_{\delta}$ lies in the corresponding copy of $\mathbb{S}^n$. Then there exist pairwise 
disjoint functions $u_1,\ldots, u_m$ such that $R_{g_n}[u_i]<0$ for all $i$. Thus, the same is true 
for any linear combination of these functions, i.e. $\lambda_m(P_{1,g_n}^R)\leqslant 0$.


\section{Curvature prescription}\label{sec:curvature:prescription}

\subsection{Background}  A natural extension of the Yamabe problem on a manifold $M$ with boundary 
$\bd M$ is to
prescribe the scalar curvature $R_g$ in the interior of $M$ and the mean curvature $h_g$ on $\bd M$.   
Existence and uniqueness for this problem was considered e.g. in the papers \cite{Escobar:1996:CV,Escobar:2003:JFA}.  In the paper \cite[\S 5]{CGJP2}, the authors obtained several results that give 
obstructions 
to prescribing curvature in terms of eigenfunctions lying in the kernel of the conformal Laplacian.   
In this section, we generalize those results to manifolds with boundary.

We start with the following result: 
\begin{theorem}
\label{prescription1}
For any $u\in\ker P_{1,g}^R$ one has the following identity
$$
\frac{n-2}{4(n-1)}\int\limits_M R_g u \,dv_g + \frac{n-2}{2(n-1)}\int\limits_{\bd M} h_g u \,d\sigma_g = 0.
$$
\end{theorem}

\begin{proof}
By definition one has $\frac{n-2}{4(n-1)}R_g = P_{1,g}(1)$. Integrating by parts we obtain
$$
\int\limits_M P_{1,g}(1)u\,dv_g = \int\limits_M P_{1,g}(u)\,dv_g + 
\int\limits_{\bd M}\bd_n u\,d\sigma_g = -\frac{n-2}{2(n-1)}\int_{\bd M} h_g u\,d\sigma_g.
$$
\end{proof}

This theorem has consequences for the curvature prescription problem which we describe below. 
For a fixed conformal class of metrics $\mC = [g]$ one defines the following curvature map
$$
\mR\colon \mC\to C^\infty(M)\times C^\infty(\partial M),\qquad \mR(g) = (R_g,h_g).
$$
Our goal is to determine the image $\mR(\mC)$. It turns out that non-trivial elements $u\in 
\ker(P_{1,g}^R)$ constitute obstructions to this curvature prescription problem. 
An easy example is the following corollary.

\begin{corollary}
For any $u\in\ker P_{1,g}^R \setminus\{0\}$, the pair $(u,u|_{\bd M})$ does not belong to $\mR(\mC)$.
\end{corollary}
\begin{proof}
By Theorem~\ref{prescription1} and formula~\eqref{kernel:change}, for any metric 
$\hat g = e^{2\omega}g$ one has
$$
\frac{n-2}{4(n-1)}\int\limits_M R_{\hat g} e^{-(n/2-1)\omega} u \,dv_{\hat g} + 
\frac{n-2}{2(n-1)}\int\limits_{\bd M} h_{\hat g} e^{-(n/2-1)\omega} u \,d\sigma_{\hat g}=0.
$$
Assume the contrary, i.e. $\mR(\hat g) = (u,u|_{\bd M})$. Then 
the previous formula implies
$$
\frac{n-2}{4(n-1)}\int\limits_M e^{-(n/2-1)\omega} u^2 \,dv_{\hat g} + 
\frac{n-2}{2(n-1)}\int\limits_{\bd M} e^{-(n/2-1)\omega} u^2 \,d\sigma_{\hat g}=0,
$$
which is impossible.
\end{proof}

Similarly, any pair that has the same or the opposite strict signs as $(u,u|_{\bd M})$ can not belong to $\mR(\mC)$.

\begin{corollary}
Let $M$ be a manifold with boundary equipped with a conformal class $\mC$, let $g\in \mC$ 
and $u\in \ker P_{1,g}^R\backslash \{0\}$. Then any pair $(Q,f)$ such that $(Q,f)$ 
has the same or the opposite strict sign as $(u,u|_{\bd M})$ can not belong to $\mR(\mC)$.
\end{corollary}

Examples of such pairs include $(e^{\omega_1}u^{p_1},  e^{\omega_2}u^{p_2}|_{\bd M})$ for 
any pair of odd natural numbers $(p_1,p_2)$ and $(\omega_1,\omega_2)\in C^\infty(M)\times C^\infty(\bd M)$.

\subsection{Obstructions from nodal domains}
Next, we would like to generalize \cite[Thm. 5.5, Cor. 5.6]{CGJP2} to manifolds with boundary.  
There is 
an extensive literature studying the relationship between the spectral theory 
of conformally covariant operators on a manifold $M$, and the behaviour of 
a related conformally covariant operators on $\bd M$, see e.g. \cite{GZ,Case,Chang:Gonzalez,GP-18}.  The following discussion deals with the relationship between nodal sets of 
eigenfunctions 
of a conformally covariant operator $P$ on $M$, and the nodal sets of 
eigenfunctions of a related operator $Q$ on $\bd M$.  Clearly, if the nodal set $\cN(u)$ 
of an eigenfunction $u$ of $P$ intersects the boundary, the limiting value of $u$ on $\bd M$ 
is equal to $0$. Conversely, let us suppose that $u(x)$ is an eigenfunction of a conformally 
covariant BVP, and that $u(y)=0$ for some $y\in\bd M$.  A natural question is whether 
there exists a sequence $x_i\in \operatorname{int}M$ such that $u(x_i)=0$ and $x_i \to y$. In other words:  
is a nodal set on the boundary always a limit of an interior nodal set? This is the case by a 
simple application of the strong maximum principle. Indeed, assume the contrary, i.e. that $u(y) = 0$ 
but $u(x)>0$ for $x$ in the pointed neighbourhood of $y$. Then by Hopf's lemma (see, 
e.g.~\cite[Theorem 2.8(b)]{K}) one has $\bd_\nu u(y)<0$. At the same time, the boundary condition 
$B_g(u)(y) = \bd_\nu u(y) + \frac{n-2}{2(n-1)}h_gu(y) = 0$ yields $\bd_\nu u(y) = 0$, a contradiction.

\vspace{0.3cm}
We now describe the obstruction in question. We prove

\begin{theorem}
\label{nodal1}
Let $u\in \ker(P_{1,g}^R)$ and let $\Omega$ be a nodal domain of $u$. Then, for all $v\in C^\infty(M)$
\begin{equation}
\int\limits_\Omega |u|P_{1,g}(v) \,dv_g = - \int\limits_{\bd\Omega\backslash \bd M} v 
|\nabla u|_g\, d\sigma_g - \int\limits_{\bd\Omega\cap\bd M} |u|B_g(v)\,d\sigma_g.
\end{equation}
\end{theorem}
\begin{proof}
Up to a change of sign, we can assume that $u>0$ on $\Omega$. Integrating by parts we obtain
$$
\int\limits_\Omega uP_{1,g}(v) \,dv_g = \int\limits_\Omega vP_{1,g}(u) \,dv_g - 
\int\limits_{\bd\Omega} (u\bd_nv - v\bd_nu)\,d\sigma_g.
$$

Decomposing the second summand on the right hand side into the sum of 
integrals over $\bd\Omega\cap\bd M$ and $\bd\Omega\backslash\bd M$ we obtain
$$
\int\limits_{\bd\Omega\cap\bd M} (u\bd_nv - v\bd_nu)\,d\sigma_g = 
\int\limits_{\bd\Omega\cap\bd M}u\left (\bd_n v + \frac{n-2}{2(n-1)}h_gv\right)\,d\sigma_g,
$$ 
$$
\int\limits_{\bd\Omega\backslash\bd M} (u\bd_nv - v\bd_nu)\,d\sigma_g = 
-\int\limits_{\bd\Omega\backslash\bd M} v\partial_nu\,d\sigma_g.
$$
Since on $\bd\Omega\backslash\bd M$ one has $\bd_n u = -|\nabla u|_g$ the proof is complete.
\end{proof}

\begin{theorem}
Let $(Q,f)\in \mR([g])$ be the scalar curvature and mean curvature of some metric in 
the conformal class $[g]$. Then, there is a pair of positive functions $\omega_i\in 
C^\infty(M)$, $\omega_b\in C^\infty(\bd M)$ such that for any $u\in\ker(P_{1,g}^R)$ 
and any nodal domain $\Omega$ of $u$,
$$
\int\limits_\Omega Q|u|\omega_i \,dv_g + \int\limits_{\bd\Omega\cap\bd M}f|u|\omega_b\,d\sigma_g<0.
$$
\end{theorem}
\begin{proof}
By Theorem~\ref{nodal1} one has
$$
\int\limits_\Omega |u|P_{1,g}(v) \,dv_g + \int\limits_{\bd\Omega\cap\bd M} |u|B_g(v)\,d\sigma_g <0
$$
for any positive function $v \in C^\infty(M)$.
Let $\hat g = e^{2\omega}g$ be the metric such that $\mR(\hat g) = (Q,f)$. Let us apply 
the previous inequality for $v =1$ and the metric $\hat g$, taking into account the transformation 
law~\eqref{kernel:change}
\begin{equation*}
\frac{n-2}{4(n-1)}\int\limits_\Omega e^{-(n/2-1)\omega}|u|Q \,dv_{\hat g} +
\frac{n-2}{2(n-1)}\int\limits_{\bd\Omega\cap\bd M} e^{-(n/2-1)\omega}|u|f\,d\sigma_{\hat g}<0.
\end{equation*}
Equivalently,
$$
\frac{n-2}{4(n-1)}\int\limits_\Omega 
e^{(n/2+1)\omega}|u|Q \,dv_g +
\frac{n-2}{2(n-1)}\int\limits_{\bd\Omega\cap\bd M} e^{\frac{n+1}{2}\omega}|u|f\,d\sigma_g<0.
$$
Setting $\omega_i = \frac{n-2}{4(n-1)}e^{(n/2+1)\omega}$ and 
$\omega_b = \frac{n-2}{2(n-1)}e^{\frac{n+1}{2}\omega}$ completes the proof.
\end{proof}

\begin{corollary}
For any $(Q,f)\in\mR([g])$ and any nodal domain $\Omega$ of 
$u\in\ker(P_{1,g}^R)$, either $Q$ or $f$ must be negative for at least one point of $\bar\Omega$. 
\end{corollary}

\begin{theorem}
Let $u\in\ker P^R_{1,g}$ and let $\Omega$ be a nodal domain of $u$. Then for any $
\hat g = e^{2\omega}g \in[g]$ one has
$$
-\int\limits_{\bd\Omega\backslash\bd M} 
e^{\left(1 - \frac{n}{2}\right)\omega}|\nabla\hat u|_{\hat g}\,d\sigma_{\hat g} = 
\frac{n-2}{4(n-1)}\int\limits_{\Omega} |u|R_g\,dv_g + \frac{n-2}{2(n-1)}
\int\limits_{\bd\Omega\cap\bd M} h_g |u|\,d\sigma_g,
$$
i.e. the expression on the left hand side is a conformal invariant.
\end{theorem}
\begin{proof}
Applying Theorem~\ref{nodal1} for the metric $\hat g$ and $v = e^{\left(1 - \frac{n}{2}\right)\omega}$ yields
$$
-\int\limits_{\bd\Omega\backslash\bd M} 
e^{\left(1 - \frac{n}{2}\right)\omega}|\nabla\hat u|_{\hat g}\,d\sigma_{\hat g} =
\int\limits_\Omega |\hat u| P_{1,\hat g}(v)\,dv_{\hat g} + 
\int\limits_{\bd\Omega\cap\bd M}|\hat u| B_{\hat g}(v)\,d\sigma_{\hat g}.
$$
By the transformation laws for $P_{1, g}$ and $B_{g}$ one has $P_{1,\hat g}(v) = 
e^{-(n/2+1)\omega}P_{1,g}(1)$ and $B_{\hat g} = e^{-\frac{n}{2}\omega}B_g(1)$. Therefore one has
\begin{equation*}
\begin{split}
-\int\limits_{\bd\Omega\backslash\bd M} 
e^{\left(1 - \frac{n}{2}\right)\omega}|\nabla\hat u|_{\hat g}\,d\sigma_{\hat g} = & 
\int\limits_\Omega |e^{-(n/2-1)\omega }u| e^{-(n/2+1)\omega}P_{1, g}(1)\,e^{n\omega}dv_{g} + \\
& \int\limits_{\bd\Omega\cap\bd M}|e^{-(n/2-1)\omega } u| 
e^{-\frac{n}{2}\omega}B_{g}(1)\,e^{(n-1)\omega}d\sigma_{g}.
\end{split}
\end{equation*}
\end{proof}

\begin{remark}

Suppose there is a function $u\in \ker P_{1,g}$ (not necessarily $\ker P_{1,g}^R$!) 
such that it has a nodal domain $\Omega \Subset M$. Then by the maximum principle 
for Schr\"odinger operators, for any $(Q,f)\in \mR([g])$ the function $Q$ must change 
sign inside $\Omega$.

\end{remark}


\section{Inequalities between Dirichlet and Neumann eigenvalues}\label{sec:DNinequalities}
In the paper \cite{Friedlander}, Friedlander obtained inequalities between Dirichlet and Neumann eigenvalues of the Laplacian using the 
equality
\begin{equation}\label{Mazzeo:difference}
N_N(\lambda)-N_D(\lambda)=N_-(\cD(\lambda)),
\end{equation}
where $N_N(\lambda)$ (resp. $N_D(\lambda)$) denotes the Neumann (resp. Dirichlet) 
eigenvalue counting function,  $\cD(\lambda)$ denotes the Dirichlet-to-Neumann 
operator, 
and $N_-(\cD(\lambda))$ denotes the number of its negative eigenvalues. This equality was given a geometric interpretation by Mazzeo in \cite{Mazzeo}.

We prove a Friedlander-type result for the conformal Laplacian. Let $N_R$ and $N_D$ denote the number of negative eigenvalues for the conformal Laplacian with Robin and Dirichlet boundary conditions, respectively, and
%
recall the definition \eqref{conformal:DtoN} of the conformal Dirichlet-to-Neumann operator
$\widehat\cD_g$.

\begin{theorem}\label{prop:negative:eig:count}
$N_R - N_D = N_-(\widehat\cD_g) + \dim\ker P_{1,g}^D$.
\end{theorem}

\begin{proof}
For each $s \in [0,\infty)$ consider the operator $P_{1,g}(s)$ defined to be the conformal Laplacian $P_{1,g}$ 
with boundary conditions
\[
	B_g u + s u = 0.
\]
This is the unbounded operator on $L^2(M)$ generated by the bilinear form
\[
	a_s(u,v) = \int_M \left(\nabla u \cdot \nabla v + \frac{n-2}{4(n-1)} R_g uv\right) dv_g + \int_{\bd M} \left(\frac{n-2}{2(n-1)} h_g + s\right) uv \, d\sigma_g
\]
with form domain $H^1(M)$. From this one sees that each $P_{1,g}(s)$ is self-adjoint and the 
family is analytic with respect to $s$.

We show that the eigenvalues of $P_{1,g}(s)$ are strictly increasing. Suppose 
$\lambda = \lambda(s)$ is an analytic curve of eigenvalues, with normalized eigenfunctions 
$u = u(s)$, so that $a_s(u,v) = \lambda \left<u,v\right>$ for all $v \in H^1(M)$. 
Substituting $v = u'$ yields 
\[
	a_s(u,u') = \lambda \left<u,u'\right>.
\]
On the other hand, differentiating and then letting $v=u$ yields
\[
	a'_s(u,u) + a_s(u',u) = \lambda \left<u',u\right> + \lambda' \|u\|^2
\]
and hence
\[
	\lambda' = a'_s(u,u) = \int_{\bd M} u^2,
\]
which is strictly positive. (If $u$ vanishes on $\bd M$, the boundary condition would imply 
that $\bd_\nu u$ also vanishes, which contradicts the unique continuation principle.)

Note that $P_{1,g}(0) = P_{1,g}^R$ has $N_R$ negative 
eigenvalues. On the other hand, the proof of \cite[Theorem 2.4]{AM1} implies that the $k$th ordered 
eigenvalue $ \lambda_k(s)$ tends to the $k$th Dirichlet eigenvalue $\lambda^D_k$ as 
$s \to \infty$. This implies $P_{1,g}(s)$ has 
$N_D + \dim\ker P_{1.g}^D$ negative eigenvalues for all sufficiently large $s$.

Therefore, the difference $N_R - N_D - \dim\ker P_{1,g}^D$ equals the number of eigenvalue curves that have 
$\lambda(s) = 0$ for some finite $s > 0$. This occurs when there is a nontrivial solution to the Robin boundary value problem
\begin{align}\label{RobinBVP}
\begin{split}
	P_{1,g} \psi = 0 \text{ in } M\\
	B_g \psi + s \psi = 0 \text{ on } \bd M.
\end{split}
\end{align}
If 0 is not a Dirichlet eigenvalue, this is equivalent to $\cD_g u + s u = 0$, where $u = \psi\big|_{\bd M}$. Moreover, the multiplicity of $\lambda =0$ as an eigenvalue of the above Robin problem equals the multiplicity of $-s$ as an eigenvalue of $\cD_g$.

When 0 is a Dirichlet eigenvalue we have to consider $\widehat\cD_g$ instead of $\cD_g$. If $\psi$ solves \eqref{RobinBVP}, it must be the case that $u = \psi\big|_{\bd M} \in S$, hence 
\[
	\widehat\cD_g u = \Pi_S B_g \psi = - s \Pi_S u = - s u.
\]
Therefore $-s$ is an eigenvalue of $\widehat\cD_g$. If $\psi_1$ and $\psi_2$ are 
linearly independent solutions to \eqref{RobinBVP}, the corresponding functions $u_1$ and 
$u_2$ on the boundary must be linearly independent as well. Otherwise, there would exist a 
linear combination $\psi = c_1 \psi_1 + c_2 \psi_2$ with $P_{1,g} \psi = 0$, $\psi\big|_{\bd M} = 0$, 
and $B_g \psi = 0$, hence $\bd_\nu \psi = 0$, which is only possible if $\psi \equiv 0$, a 
contradiction. Therefore, the multiplicity of $-s$ is at least as large as the multiplicity of 
$\lambda=0$ for \eqref{RobinBVP}.

Conversely, if $-s$ is an eigenvalue of $\widehat\cD_g$, with eigenfunction $u$, 
there exists a function $\psi$ on $M$ such that $P_{1,g} \psi = 0$ and $\Pi_S B_g \psi = - 
s u$. Given $\psi$, there exists a unique $v \in \ker P_{1,g}^D$ so that $B_g(\psi + v) 
\in S$. Then $P_{1,g}(\psi + v) = 0$ 
and
\[
	B_g(\psi + v) = \Pi_S B_g(\psi + v) = \Pi_S B_g(\psi) = - s \psi\big|_{\bd M} = 
	- s(\psi + v)\big|_{\bd M},
\]
hence 0 is an eigenvalue of \eqref{RobinBVP}, with eigenfunction $\psi + v$.

Suppose $u_1$ and $u_2$ are linearly independent eigenfunctions of $\widehat\cD_g$, 
with the same eigenvalue $-s$. If the corresponding eigenfunctions $\psi_i + v_i$ 
of \eqref{RobinBVP} are not linearly independent, there will exist constants $c_1$ and $c_2$ 
so that $c_1(\psi_1 + v_1) + c_2 (\psi_2 + v_2) = 0$.
Restricting to the boundary yields $c_1 u_1 + c_2 u_2 = 0$, a contradiction. This 
proves that the multiplicity of $\lambda =0$ for \eqref{RobinBVP} equals the multiplicity 
of $-s$ for $\widehat\cD_g$, and thus completes the proof.
\end{proof}

\begin{remark}
Repeating the above argument with $P_{1,g}$ replaced by $P_{1,g} - \lambda$, we obtain a similar equality in terms of the counting functions $N_R(\lambda)$ and $N_D(\lambda)$ for the conformal Laplacian. However, these quantities are only conformally invariant when $\lambda=0$.
\end{remark}


\section{A continuous family of GJMS operators and their conformal extensions to 
Poincare-Einstein manifolds}\label{sec:Poincare:Einstein}

In this section, we put the previous results in a much more general framework which has been developed in recent 
years. First we consider higher order extensions of the conformal Laplacian, i.e. the so-called GJMS operators 
\cite{GJMS}. The original construction of these operators did not provide useful information on their 
analytic properties, due to the nature itself of the construction. The important result on their analytic properties  
came with the paper by Graham and Zworski \cite{GZ}.   
Roughly speaking, they prove that if $(M^n,[g])$ is a smooth compact manifold
endowed with a conformal structure, then the GJMS operators $P_k$ can be realized as 
residues at the values $\gamma = k$ of
the meromorphic family $S(n/2 + \gamma)$ of scattering operators associated to the Laplacian 
on any Poincar\'e-Einstein manifold
$(X,G)$ for which $(M^n,[g])$ is the conformal infinity.   These are the `trivial' poles of the scattering 
operator $S(s)$; typically $S(s)$ has infinitely many 
other poles, which are
called resonances. In this geometric framework, 
multiplying this scattering family by some regularizing factor, one obtains a holomorphic 
family of
elliptic pseudodifferential operators denoted $P_\gamma^{g}$ (which patently depends on the filling 
$(X,G)$) for $\gamma \in (0,n/2)$. This realization of the GJMS operators (and their continuous in $\gamma$ counterpart) has led to 
important new understanding of them, including, for example, the basic fact that
$P_\gamma^{g}$ is symmetric with respect to $dv_{g}$. Hence even though the family $P_\gamma^{g}$ is not entirely canonically 
associated to $(M,[g])$, its study can still illuminate the truly canonical operators which occur 
as special values, i.e. the GJMS operators.

Let us sum up the main properties of these operators: $P^g_0 = \mathrm{Id}$, and more generally, $P^g_k$ 
is the $k^{\mathrm{th}}$ GJMS operator; $P^g_\gamma$ is a classical elliptic pseudodifferential operator 
of order $2\gamma$
with principal symbol $\sigma_{2\gamma}(P_\gamma^{g}) = |\xi|^{2\gamma}_{g}$; $P^g_\gamma$
is Fredholm on $L^2$ when $\gamma > 0$; if $P^g_\gamma$ is invertible, then 
$P^g_{-\gamma} = (P^g_{\gamma})^{-1}$; and, most importantly,
\begin{equation}
\mbox{if}\ \hat g=u^{\frac{4}{n-2\gamma}} g, \qquad \mbox{then}\ 
P_\gamma^{g} (uf) = u^{\frac{n+2\gamma}{n-2\gamma}} P_\gamma^{\hat g} (f)
\label{eq:ccfcL}
\end{equation}
for any smooth function $f$. Generalizing the formul\ae\ for scalar curvature ($\gamma = 1$) 
and the Paneitz-Branson
$Q$-curvature ($\gamma = 2$), for any 
$0 < \gamma < n/2$ we define the
$Q$-curvature of order $\gamma$ associated to a metric $g$ by
\begin{equation}
Q_\gamma^{g} = P_\gamma^{g}(1).
\label{eq:Qgamma}
\end{equation}

It is then natural to define a ``generalized Yamabe problem": given a metric $g$  on a compact
manifold $M$, find $u > 0$ so that if $\hat g = u^{4/(n-2\gamma)} g$, then $Q_\gamma^{\hat g}$ is 
constant. This amounts to solving
\begin{equation}\label{fracYamabe}
P_\gamma^{g} u = Q_\gamma^{\hat g} u^{\frac{n+2\gamma}{n-2\gamma}} ,\quad u>0,
\end{equation}
for $Q_\gamma^{\hat g} = \mbox{const.}$  From the analytic point of view, this problem exhibits 
the same features as the standard Yamabe problem (lack of compactness, for instance) with 
the additional major difficulty that the equation involves a pseudodifferential operator. 

In the last decade or so, several works have studied this generalized Yamabe problem. 
The first major tool is a result of Chang and Gonzalez \cite{Chang:Gonzalez} realizing the operators $P_\gamma^g$ for 
$\gamma \in (0,1)$ as boundary operators, i.e. Dirichlet-to-Neumann type operators. It relies 
heavily on the original construction based on scattering of Graham and Zworski, reformulating it in a suitable form. 
However, a drawback of the Chang-Gonzalez result is that the operator defined on the 
Poincar\'e-Einstein manifold $(X,G)$
is {\sl not} conformally covariant except in the case $\gamma=1/2$, which corresponds 
precisely to the Escobar problem we considered in the previous sections. This was later resolved by 
Case and Chang \cite{CaseChang},
where they reinterpret the scattering theory of Graham and Zworski in terms of smooth 
metric measure spaces. We now introduce this framework. 

\subsection*{Smooth metric measure spaces and conformally covariant operators of 
order $2\gamma$ with $\gamma \in (0,1)$} We will refrain from going deeply in the 
theory of Case and Chang since it is quite technical and would obscure our goal of applying their construction 
to the spectral geometry of their operators. Notice first that the most 
popular range of powers is $\gamma \in (0,1)$ since in this case the problem enjoys a 
standard maximum principle. However, as already noted, the case $\gamma=2$ is also of major importance since 
it corresponds to the Paneitz-Branson operator. We will thus focus on the range $\gamma \in (0,1)$, 
but several results mentioned 
in this section can be obtained for a wider range (under appropriate assumptions). We now describe their construction. 
Notice that we change slightly the notations below to be consistent with the previous discussion starting this section. 

A smooth metric measure space is a four-tuple $(\overline{X}^{n+1},G,v^m \text{dvol},\mu)$ 
determined by a Riemannian manifold $(\overline{X},G)$ with boundary $M=\partial X$, the 
Riemannian volume element $\text{dvol}$ associated to $G$, a nonnegative function 
$v \in C^\infty(\overline X)$ with $v^{-1}(0)=M$ and constants $m \in \mathbb R 
\backslash \left \{ 1-n \right \}$ and $\mu \in \mathbb R$. In the interior 
of $\overline X$ we define a smooth function $\phi$ by 
$$
v^m=e^{-\phi}.
$$
We say that $(\overline{X}^{n+1},G,v^m \text{dvol}_G,\mu)$ and 
$(\overline{X}^{n+1},\hat G,{\hat v}^m \text{dvol}_{\hat G},\mu)$ are pointwise conformally 
equivalent if there exists a smooth function $u$ such that 
$$
\hat g=u^{-2}g
$$ 
 and
 $$
 \hat v= u^{-1}v.
 $$
 We then define the weighted Laplacian (in relation to the 
 Bakry-Emery Ricci tensor) and the weighted 
GJMS operator of order $2$ by
$$
\Delta_\phi=\Delta-\nabla \phi
$$
and 
$$
L^m_{2,\phi}=-\Delta_\phi +\frac{m+n-1}{2} J^m_\phi,
$$
where 
$$
J^m_\phi=\frac{1}{2(m+n)} (R-2m v^{-1}\Delta v-m-(m-1)v^{-2}|\nabla v|^2+m\mu v^{-2}). 
$$
The important property for us is conformality in the following sense: if $\hat g =u^{-2}g $ and 
$\hat v =u^{-1} v$, then 
$$
\widehat{L^m_{2,\phi}} w=u^{(m+n+3)/2} L^m_{2,\phi}(u^{-(m+n-1)/2}w)
$$
for every smooth $w$. We can now state the main theorem (see \cite{CaseChang}, 
Theorem 4.1 and Lemma 6.1) in the case $\gamma \in (0,1)$.
\begin{theorem}
Let $(X^{n+1}, M^n,G_+)$ be a Poincare-Einstein manifold and $\gamma \in (0,1)$. Suppose that 
$\lambda_1(-\Delta_{G_+})>\frac{n^2}{4}-\gamma^2$ and set $m_0=1-2\gamma$. Fix a representative $h$ of
 the conformal boundary and let $r$ be the geodesic defining function associated to $h$. Then there exists a unique defining function $\rho$ such that  
$$\rho=r+c_\gamma Q_{\gamma}r^{1+2\gamma}+O(r^3)$$
where $c_\gamma$ is just a normalizing constant and 
$$
J^{m_0}_{\phi_0}=0. 
$$

Consider now the smooth metric measure space $(\overline{X}^{n+1}, 
G:=\rho^2 G_+,\rho^{m_0}\text{dvol},m_0-1).$ Given $f$ smooth, the function $U$ is the solution of 
\begin{equation}\label{Dirichlet:scat}
\left\{
\begin{aligned}
L^{m_0}_{2,\phi_0} U &=0, \quad x\in X\\
U &= f, \quad x\in M
\end{aligned}
\right.
\end{equation}
if and only if $u=\rho^{n-s} U$ is the solution of 
\begin{equation}\label{Dirichlet:ext}
\left\{
\begin{array}{c}
-\Delta_{G_+}u-s(n-s) u =0, \quad x\in X\\
u=Fr^{n-s}+Gr^s\\
F|_{r=0} = f, \quad x\in M
\end{array}
\right.
\end{equation}
for $s=\frac{n}{2}+\gamma$. Moreover $U$ is such that 
$$
P_{\gamma} f -\frac{n-2\gamma}{2}c_\gamma Q_{\gamma} f=\frac{c_\gamma}{2\gamma} \lim_{\rho \to 0} 
\rho^{m_0} \partial_\rho U. 
$$
\end{theorem}

Applying the previous theorem, the generalized Yamabe problem \eqref{fracYamabe} 
for $\gamma \in (0,1)$ is equivalent to solving the BVP 
 
\begin{equation}\label{Dirichlet:scat:2}
\left\{
\begin{aligned}
-\Delta_{\phi_0} U &=0, \quad x\in X\\
\lim_{\rho \to 0}\rho^{1-2\gamma} \partial_\rho U-\frac{n-2\gamma}{2}c_\gamma Q_{\gamma} U &= 
d_\gamma U^{\frac{n+2\gamma}{n-2\gamma}}, \quad x\in M
\end{aligned}
\right.
\end{equation}
for some harmless constants $c_\gamma$ and $d_\gamma$ (with signs as defined in \cite{CaseChang}). 
For $\gamma=1/2$ this is exactly the Escobar problem \cite{Escobar:1992:Annals}. Notice that by Hopf's 
lemma (see \cite{CaseChang}), one can prove that the problem 
\eqref{Dirichlet:scat:2} admits a strong maximum principle. The  point of the BVP \eqref{Dirichlet:scat:2} is that, by integration by parts on the weighted smooth 
space and the conformality of the operators, one can deduce obstructions by nodal domains 
for curvature prescription problems in exactly the same way as in section \ref{sec:curvature:prescription}. Similarly, one can define a version of the conformal Robin BVP and obtain analogues of the whole set of previous results, adapting 
straightforwardly the same computations.  In particular, having in mind the program for the conformal Laplacian 
on manifolds with boundary, one could study the analogue of the Kazdan-Warner trichotomy. 

The previous construction is developed here in the case $\gamma \in (0,1)$. Case and Chang used 
the setting of smooth metric measure spaces to give the analogue of the previous theorem for the {\sl whole} family 
of operators $P^g_\gamma$. However, in that case the formulas are much more involved and it starts being cumbersome 
to report  the results in this section. 


\section{Further questions}\label{sec:further:questions}

In this section, we formulate several open problems that seem interesting to the 
authors, continuing the list in \cite[\S 7]{CGJP2}.

\begin{itemize}
\item[(A)] We showed in section \ref{sec:no:zero} that generically, $0$ is not an 
eigenvalue of the conformal Laplacian on manifolds with boundary, for the natural boundary 
value problems considered in this paper.  The following  
is a natural analogue of Conjecture C in \cite[\S 7.3]{CGJP1} for manifolds with boundary: 
among metrics where $0$ {\em is} an eigenvalue of the conformal Laplacian, generically 
its multiplicity is equal to $1$.  
\item[(B)] Relative conformal invariants: Can the results in \cite{Schippers} 
be generalized from dimension $2$ to higher dimensions?  
\end{itemize}

It is known that on some closed manifolds (\cite{GL,Ros06}) the 
space of Yamabe-positive metrics can have infinitely many connected components.
On the other hand, the space of Yamabe-negative metrics is connected and has trivial 
homotopy groups (\cite{Lo92,Kat}).  The number of negative eigenvalues of the 
conformal Laplacian for one of the conformally invariant boundary value problems 
considered in section \ref{sec:BVP} provides a natural grading of the space of 
Yamabe-negative metrics; this leads naturally to Problem K raised in \cite[\S 7]{CGJP2}.  
Below, we formulate its natural counterpart  for manifolds with boundary.  

\begin{itemize}
\item[(C)]  Let $[g_0]$ and $[g_1]$ 
be two conformal classes of metrics on $(M,\bd M)$ such that $P_{g_0}$ and $P_{g_1}$ have 
at least $k$ negative eigenvalues for one of the conformally invariant boundary value problems 
considered in section \ref{sec:BVP}.  Is it possible to connect $g_0$ and $g_1$ by a curve of metrics $g_t$ 
such that $P_{g_t}$ has at least $k$ negative eigenvalues for the corresponding boundary 
value problems?
\end{itemize}


\section*{Acknowledgements}
The authors would like to thank Asma Hassannezhad and A. Rod Gover for stimulating discussions and 
very interesting remarks about preliminary versions of this paper.  In addition, the authors would like to thank  
Ailana Fraser, Pengfei Guan and Richard Schoen for useful discussions.  The authors would like to thank 
BIRS, CRM, Oberwolfach, McGill and Johns Hopkins University for their hospitality.  
 

\end{document}